\documentclass[11pt]{article}
 \usepackage{amssymb,latexsym,amsmath,amsthm}
\newtheorem{thm*}{Theorem}[section]

\newtheorem{thm}{Theorem}[section]

\newtheorem{dfn}[thm]{Definition}

\newtheorem{lemma}[thm]{Lemma}

\newtheorem{remark}[thm]{Remark}

\newtheorem{cor}[thm]{Corollary}

\begin{document}

\def\d{ \partial_{x_j} } 
\def\Na{{\mathbb{N}}}

\def\Z{{\mathbb{Z}}}

\def\IR{{\mathbb{R}}}

\newcommand{\E}[0]{ \varepsilon}

\newcommand{\la}[0]{ \lambda}

\newcommand{\s}[0]{ \mathcal{S}}

\newcommand{\AO}[1]{\| #1 \| }

\newcommand{\BO}[2]{ \left( #1 , #2 \right) }

\newcommand{\CO}[2]{ \left\langle #1 , #2 \right\rangle} 

\newcommand{\R}[0]{ \IR\cup \{\infty \} } 

\newcommand{\co}[1]{ #1^{\prime}} 

\newcommand{\p}[0]{ p^{\prime}} 

\newcommand{\m}[1]{   \mathcal{ #1 }} 

\newcommand{ \W}[0]{ \mathcal{W}}

%  norm of H
\newcommand{ \A}[1]{ \left\| #1 \right\|_H }

% inner product H 
\newcommand{\B}[2]{ \left( #1 , #2 \right)_H }

% H^* , H pairing
\newcommand{\C}[2]{ \left\langle #1 , #2 \right\rangle_{  H^* , H } }

 \newcommand{\HON}[1]{ \| #1 \|_{ H^1} }

% Omega    \Om
\newcommand{ \Om }{ \Omega}

% \partial Omega      \pOm 
\newcommand{ \pOm}{\partial \Omega}

%   D(\Omega)   \D
\newcommand{\D}{ \mathcal{D} \left( \Omega \right)}

% D'( Omega)        \DP 
\newcommand{\DP}{ \mathcal{D}^{\prime} \left( \Omega \right)  }

% D' pairing
\newcommand{\DPP}[2]{   \left\langle #1 , #2 \right\rangle_{  \mathcal{D}^{\prime}, \mathcal{D} }}

% (H^1)^* , H^1    (( pairing ))      \PHH
\newcommand{\PHH}[2]{    \left\langle #1 , #2 \right\rangle_{    \left(H^1 \right)^*  ,  H^1   }    }

%  H^{-1} , H_0^1  (( pairing ))   \PHO
\newcommand{\PHO}[2]{  \left\langle #1 , #2 \right\rangle_{  H^{-1}  , H_0^1  }} 

 %  H^1(\Omega)     \HO
 \newcommand{\HO}{ H^1 \left( \Omega \right)}

%  H_0^1( \Omega)       \HOO
\newcommand{\HOO}{ H_0^1 \left( \Omega \right) }

% C_c^\infty(omega)
\newcommand{\CC}{C_c^\infty\left(\Omega \right) }

%H_0^1(Omega)  norm
\newcommand{\N}[1]{ \left\| #1\right\|_{ H_0^1  }  }

%H_0^1(Omega)   innerproduct 
\newcommand{\IN}[2]{ \left(#1,#2\right)_{  H_0^1} }

% H^1(\Omega) inner product
\newcommand{\INI}[2]{ \left( #1 ,#2 \right)_ { H^1}} 

% (H^1(\Omega))^*
\newcommand{\HH}{   H^1 \left( \Omega \right)^* } 

% ( H^{-1}(\Omega))
\newcommand{\HL}{ H^{-1} \left( \Omega \right) }

\newcommand{\HS}[1]{ \| #1 \|_{H^*}}

\newcommand{\HSI}[2]{ \left( #1 , #2 \right)_{ H^*}}

\newcommand{\WO}{ W_0^{1,p}} 
\newcommand{\w}[1]{ \| #1 \|_{W_0^{1,p}}}  

\newcommand{\ww}{(W_0^{1,p})^*}   

\newcommand{\Ov}{ \overline{\Omega}} 

\title{Regularity of Extremal Solutions in Fourth Order Nonlinear Eigenvalue Problems on General Domains}

\author{{Craig Cowan\footnote{Department of Mathematics, University of British Columbia, Vancouver, B.C. Canada V6T 1Z2. E-mail: ctcowan@math.ubc.ca. This work is part of the author's PhD dissertation in preparation under the supervision of N. Ghoussoub.}
\qquad
Pierpaolo Esposito\footnote{Dipartimento di Matematica,
Universit\`a degli Studi ``Roma Tre", 00146 Roma, Italy. E-mail: esposito@mat.uniroma3.it. Research partially supported by FIRB-IDEAS (2008), project ``Geometrical aspects in PDEs".}\qquad
Nassif Ghoussoub\footnote{Department of Mathematics, University of British Columbia, Vancouver, B.C. Canada V6T 1Z2. E-mail: nassif@math.ubc.ca. Research partially supported by the Natural Science and Engineering Research Council of Canada.}
}\\ \\  
{\it Dedicated to Louis Nirenberg for his $85^{\rm th}$ birthday}}
\date{\today}
\smallbreak
\maketitle

\begin{abstract}  We examine the regularity of the extremal solution of the nonlinear eigenvalue problem $\Delta^2 u = \lambda f(u)$ on a general  bounded domain $\Omega$ in $ \IR^N$, 
with the Navier boundary condition $ u=\Delta u =0 $ on $ \pOm$.  Here $ \lambda$ is a positive parameter and $f$ is a non-decreasing nonlinearity with $f(0)=1$.
We give general pointwise bounds and energy estimates  which show that for any convex and superlinear nonlinearity $f$, the extremal solution $ u^*$ is smooth provided $N\leq 5$.   
\begin{itemize}
%\item  $f$ is a general convex and superlinear nonlinearity and $N\leq 5$. 
\item  If in addition $\liminf\limits_{t \to +\infty}\frac{f (t)f'' (t)}{(f')^2(t)}>0$, then $u^*$ is regular for $N\leq 7$.
\item On the other hand, if  $\gamma:= \limsup\limits_{t \to +\infty}\frac{f (t)f'' (t)}{(f')^2(t)}<+\infty$, then the same holds for $N < \frac{8}{\gamma}$.
\end{itemize}
It follows that $u^*$ is smooth if  $f(t) = e^t$ and $ N \le 8$,  or if  $f(t) = (1+t)^p$ and 
$N< \frac{8p}{p-1}$.
%either  $1 < p < \infty$ and $  N \le 8$ or if $ N \ge 9$  and $ 1 < p < \frac{N}{N-8}$.
We also show that if $ f(t) =  (1-t)^{-p}$, $p>1$ and $p\neq 3$, then $u^*$ is smooth for $N \leq  \frac{8p}{p+1}$. 
%(iv) $ f(t) = \frac{1}{(1-t)^{2}}$ and $ N \le 5$.  \\
 We note that while these are major  improvements on what is known for general domains,   they still fall short of the expected optimal results as recently established for Dirichlet problems on radial domains, e.g., $ u^*$ is smooth for $ N \le 12$ when $ f(t) = e^t$ \cite{DDGM}, and for $ N \le 8$  when     $ f(t) = (1-t)^{-2}$ \cite{CEG} (see also \cite{MoradifamMEMS}).

\end{abstract}

\section{Introduction} 
We examine the problem 
\begin{equation*} 
\tag{$N_\lambda$} \left\{ 
\begin{array}{ll}
\Delta^2 u = \lambda f(u) & \hbox{in }\Omega \\
u =\Delta u = 0 &\hbox{on } \pOm, 
\end{array}
\right.
\end{equation*} 
where $ \lambda \ge 0$ is a parameter, $ \Omega$ is a bounded domain in $ \IR^N$, $N\geq 2$, and where $f$ satisfies  one of the following two conditions:\\

(R): \qquad $f$ is smooth, increasing, convex on  $\IR$ with $ f(0)=1$ and $ f$ is superlinear at $ \infty$ (i.e. $ \displaystyle \lim_{t \rightarrow \infty} \frac{f(t)}{t}=\infty$);\\

(S): \qquad $f$ is smooth, increasing, convex on $[0, 1)$ with $ f(0)=1$ and $\displaystyle  \lim_{t \nearrow 1} f(t)=+\infty$.  \\  

Our main interest is in the regularity of the extremal solution $ u^*$ associated with $ (N_\lambda)$.    Before we discuss some known results concerning the problem $(N_\lambda)$ we recall various facts concerning second order versions of the above problem.

\subsection{The second order case} 
For a nonlinearity $ f$ of type (R) or (S), the following  second order analog of $(N_\lambda)$ with Dirichlet boundary conditions
\begin{equation*}
\tag{$Q_\lambda$} \left\{ 
\begin{array}{ll}
-\Delta u =\lambda f(u) &\hbox{in }\Omega \\
u =0 &\hbox{on } \pOm
\end{array}
\right.
\end{equation*} 
is by now quite well understood whenever $ \Omega$ is a bounded smooth domain in $ \IR^N$. See, for instance, \cite{BV,Cabre,CC,EGG,GG,Martel,MP,Nedev}. We now list the  properties one comes to expect when studying $(Q_{\lambda})$.  

\begin{itemize} \item  There exists a finite positive critical parameter $ \lambda^*$ such that for all $ 0< \lambda < \lambda^*$ there exists a \textbf{minimal solution} $ u_\lambda$ of $ (Q_{\lambda})$.   By minimal solution, we mean here that if $ v$ is another solution of $ (Q_{\lambda})$ then $v \ge u_\lambda$ a.e. in $ \Omega$.  

\item For each $ 0< \lambda < \lambda^*$ the minimal solution $ u_\lambda$ is \textbf{semi-stable} in the sense that 
\[ \int_\Omega \lambda f'(u_\lambda) \psi^2 dx \le \int_\Omega | \nabla \psi|^2 dx, \qquad \forall \psi \in H_0^1(\Omega),\]  
and is unique among all the weak semi-stable solutions. 
 
\item The map $ \lambda \mapsto u_\lambda(x)$ is increasing on $ (0,\lambda^*)$ for each $ x \in \Omega$.    This allows one to define $ u^*(x):= \lim_{\lambda \nearrow \lambda^*} u_\lambda(x)$, the so-called {\bf extremal solution}, which can be shown to be a weak solution of $ (Q_{\lambda^*})$.    In addition one can show that $ u^*$ is the unique weak solution of $(Q_{\lambda^*})$. See \cite{Martel}. 
\item There are no solutions of $ (Q_{\lambda})$ (even in a very weak sense) for $ \lambda > \lambda^*$.   

\end{itemize} 

A question which has attracted a lot of attention is whether the extremal function $ u^*$ is a classical solution of $(Q_{\lambda^*})$. This is of interest since one can then apply the results from \cite{CR}  to start a second branch of solutions emanating from $(\lambda^*, u^*)$.    Note that in the case where $ f$ satisfies (R)  (resp. (S)) it is sufficient --in view of standard elliptic regularity theory-- to show that $ u^*$ is bounded  (resp.  $\sup_\Omega u^* <1$). 

 This turned out to depend on the dimension, and so  given a nonlinearity $f$,  we say that $ N$ is the associated \textbf{critical dimension} provided the extremal solution $ u^*$ associated with $ (Q_{\lambda^*})$ is a classical solution for any bounded smooth domain $ \Omega \subset \IR^M$ for any $ M \le N-1$, and if there exists a domain $ \Omega \subset \IR^N$ such that the associated extremal solution $ u^*$ is not a classical solution.    We now list some of the known results with regard to this question.  
 
 \begin{itemize} \item  For $ f(t) = e^t$,  the critical dimension is $ N=10$. For $ N \ge 10$,  one can show that on the unit ball the extremal solution is explicitly given by $ u^*(x)= - 2 \log(|x|)$.  
 
 \item For $ \Omega =B$ the unit ball in $ \IR^N$, $ u^*$ is bounded for any $ f$ satisfying (R) provided $ N \le 9$, which --in view of the above-- is optimal (see \cite{CC}).

 \item On general domains, and if $f$ satisfies (R), then $ u^*$ is bounded for $ N \le 3$ \cite{Nedev}.   Recently this has been improved to $ N \le 4$ provided the domain is convex \cite{Cabre}.  
 
 \item For $ f(t) = (1-t)^{-2}$ the critical dimension is $N=8$ and $u^*=1-|x|^{\frac{2} {3} }$ is the extremal solution on the unit ball for $N\geq 8$.    \cite{GG}.    
  \end{itemize}    
 In the previous list, we have not considered the nonlinearity $f(t)=(t+1)^p$, $p>1$, for which the critical dimension has been also computed but takes a complicated form. The general approach to showing $N$ is the critical dimension for a particular nonlinearity $f$ is to use the semi-stability of the minimal solutions $ u_\lambda$ to obtain various estimates which translate to uniform $ L^\infty$ bounds and then passing to the limit.  These estimates generally depend on the ambient space dimension.  On the other hand, in order to show the optimality of the regularity result one generally finds an explicit singular extremal solution $u^*$ on a radial domain.  Here the crucial tool is the fact that if there exists a semi-stable singular solution in $H_0^1(\Omega)$, then it has to be the extremal solution. In practice one considers an explicit singular  solution on the unit ball and applies Hardy-type inequalities to show its semi-stability in the right dimension.  We also remark that one cannot remove the $H_0^1(\Omega)$ condition as counterexamples can be found.

\subsection{The fourth order case}      
There are two obvious fourth order extensions of $ (Q_{\lambda})$ namely the problem $(N_\lambda)$ mentioned above, and its Dirichlet counterpart 
 \begin{equation*}
\tag{$D_\lambda$} \left\{ 
\begin{array}{ll}
\Delta^2 u =\lambda f(u) &\hbox{in }  \Omega  \\
u = \partial_\nu u = 0 & \hbox{on }  \pOm, 
\end{array}
\right.
\end{equation*}   where $ \partial_\nu$ denote the normal derivative on $ \pOm$.    The problem $(Q_{\lambda})$ is heavily dependent on the  maximum principle and hence this poses a major hurdle in the study of $(D_\lambda)$ since for general domains there is no maximum principle for $ \Delta^2$ with Dirichlet boundary conditions.  But if we restrict our attention to the unit ball then one does have a weak maximum principle \cite{BOG}.      The problem $(D_\lambda)$ was studied in \cite{AGGM} and various results were obtained,  but results concerning the boundedness of the extremal solution (for supercritical nonlinearities) were missing.   

 The first (truly supercritical) results concerning the boundedness of the extremal solution in a fourth order problem are due to \cite{DDGM} where they examined the problem $(D_\lambda)$ on the unit ball in $ \IR^N$ with $ f(t)=e^t$.   They showed that the extremal solution $ u^*$ is bounded if and only if $ N \le 12$.     Their approach is heavily dependent on the fact that $\Omega$ is the unit ball. 
Even in this situation there are two main hurdles.  The first is that the standard energy estimate approach, which was so successful in the second order case, does not appear to work in the fourth order case.   The second is the fact that it is quite hard to construct explicit solutions of $(D_\lambda)$ on the unit ball that satisfy both boundary conditions, which is needed to show that the extremal solution is unbounded for $ N \ge 13$.   So what one does is to find an explicit singular, semi-stable solution which satisfies the first boundary condition, and then to perturb it  enough to satisfy the second boundary condition but not too much so as to lose the semi-stability. Davila et al. \cite{DDGM} succeeded in doing so for $ N \ge 32$, but they were forced to use a computer assisted proof to show that the extremal solution is unbounded for the intermediate dimensions $ 13 \le N \le 31$.  Using various improved Hardy-Rellich inequalities from \cite{GM} the need for the computer assisted proof was removed in \cite{Moradifam}. The case where $ f(t)=(1-t)^{-2}$ was settled at the same time in \cite{CEG}, where we used methods developed in \cite{DDGM} to show that the extremal solution associated with $(D_\lambda)$ is a classical solution if and only if $ N \le 8$.

The problem $(N_\lambda)$ with Navier boundary conditions does not suffer from the lack of a maximum principle and the existence of the minimal branch has been shown in general \cite{BG,CDG}.  If the domain is the unit ball, then again one can use the methods of \cite{DDGM} and \cite{CEG} to obtain optimal results in the case of $ f(t)=(1-t)^{-2}$ (see for instance \cite{memsbook} and \cite{MoradifamMEMS}). However, the case of a general domain is only understood in dimensions $N\leq 4$ (See \cite{GW} and \cite{memsbook}). This paper is a first attempt at giving energy estimates  on general domains, which --as mentioned above -- while they do improve known results, they still fall short of the conjectured critical dimensions that were established when the domain is a ball.

We now fix notation and some definitions associated with problem $(N_\lambda)$.

\begin{dfn}   Given a smooth solution $u$ of $ (N_\lambda)$,  we say that $u$ is a semi-stable solution of $(N_\lambda)$ if 
\begin{equation} \label{stable}
\int_\Omega \lambda f'(u) \psi^2 dx \le \int_\Omega (\Delta \psi)^2 dx, \qquad \forall \psi \in H^2(\Omega) \cap H_0^1(\Omega).
\end{equation}

\end{dfn}   

\begin{dfn}  We say a smooth solution  $ u$ of $(N_\lambda)$  is minimal  provided $ u \le v$ a.e. in $ \Omega$ for any solution $ v$ of $(N_\lambda)$. 
\end{dfn} 

We define the {\bf extremal parameter} $ \lambda^*$ as 
\[ \lambda^*:= \sup \left\{ 0 < \lambda: \mbox{there exists a smooth solution of $(N_\lambda)$} \right\}.\]

It is known, see \cite{BG,CDG,GW}, that: 
\begin{enumerate} \item  $0 < \lambda^* < \infty$.  

\item For each $ 0 < \lambda < \lambda^*$ there exists a smooth minimal solution $ u_\lambda$ of $(N_\lambda)$.  Moreover the minimal solution $ u_\lambda$ is semi-stable and is unique among the 
%weak 
semi-stable solutions.  

\item For each $ x \in \Omega$, $ \lambda \mapsto u_\lambda(x)$ is strictly increasing on $(0, \lambda^*)$, and it therefore makes sense to define  $ u^*(x):= \lim_{\lambda \nearrow \lambda^*} u_\lambda(x), $   which we call the \textbf{extremal solution}.
% It is standard to show that $u^*$  is a weak solution of $(N_{\lambda^*})$, and that
%it is the unique weak solution in a fairly broad class of solutions. 

\item  There are no solutions for $ \lambda > \lambda^*$.

\end{enumerate} 

  It is standard to show that $u^*$  is a ``weak solution" of $(N_{\lambda^*})$ in a suitable sense that we shall not define here since it will not be needed in the sequel.  One can then proceed to show that $u^*$ it is the unique weak solution in a fairly broad class of solutions. 
%\begin{remark} \label{rk11} a) To the best of our knowledge the only available  energy estimates for smooth, semi-stable solutions $u$ of $(N_\lambda)$ are given by 
%\[ \int_\Omega f'(u) u^2 dx \le \int_\Omega f(u) u dx.\] To see this, take $ \psi = u$ in (\ref{stable}) and integrate by parts $(N_\lambda)$ against $u$, and then equate.    So we see this gives $ e^{u^*} (u^*)^2 \in L^1(\Omega)$ when $f(t)=e^t$, $ (u^*+1)^p \in L^\frac{p+1}{p}(\Omega)$  when $ f(t) = (t+1)^p$ and $ (1-u^*)^{-2} \in L^\frac{3}{2}(\Omega)$  when $ f(t) = (1-t)^{-2}$.\\ 
 Regularity results on $u^*$ translate into regularity properties for any weak semi-stable solution. Indeed, by points (2)-(4) above we see that a weak semi-stable solution is either the classical solution  $u_\lambda$ or  the extremal solution $u^*$. Our preference for not stating the results in this generality is to avoid the technical details of defining precisely what we mean by a suitable weak solution. 
%\end{remark}   
     
\section{Sufficient $L^q$-estimates for regularity}
In this section, we address our attention to nonlinearities $f$ of type (R). As mentioned above, since the extremal function $u^*$ is the pointwise monotone limit of the classical solutions $u_\lambda$ as $\lambda \nearrow \lambda^*$, it suffices to consider a sequence $(u_n)_n$ of classical solutions of $(N_{\lambda_n})$, $(\lambda_n)_n$ uniformly bounded,  and try to show that 
\begin{equation} \label{bound}
\displaystyle \sup_ n \|u_n\|_\infty<+\infty.
\end{equation}
By standard elliptic regularity theory (\ref{bound}) follows by a uniform bound of $f(u_n)$ in $L^q(\Omega)$, for some $q>\frac{N}{4}$. The following result  provides a weakening of such a statement.

\begin{thm} \label{iterate} Suppose that for some $q\geq 1$ and $0 < \beta <\alpha$ we have 
\begin{equation} \label{bound1}
\sup_n  \int_\Omega f^q(u_n) <+\infty 
\end{equation} 
and 
\begin{equation} \label{bound2}
\sup_n \int_\Omega \frac{ f^\alpha(u_n)}{u_n^\beta+1} <+\infty.
\end{equation} 
 Then:
 \begin{enumerate}
  \item  If $1\leq q\leq \frac{N}{4}$ and $\alpha \leq \frac{N}{4}$, then $\displaystyle \sup_n \|f(u_n)\|_s <+\infty$ for every $s<\max\{\frac{(\alpha-\beta)N}{N-4\beta}, q\}$. 
  \item If either $q>\frac{N}{4}$ or $\alpha>\frac{N}{4}$, then $\displaystyle \sup_n \|u_n\|_\infty <+\infty$.
\end{enumerate}
\end{thm} 
 \begin{proof}  We shall first  show that under assumption (\ref{bound2}), the following holds: 
\begin{equation}\label{iteration}
\hbox{If $ \displaystyle \sup_n \|f(u_n)\|_{q_0} <+\infty$  for $ 1 \leq  q_0 \leq  \frac{N}{4}$,\,\,  then \,\, $\displaystyle \sup_n  \|f(u_n)\|_s<+\infty$ for every $s<q_1$, }
\end{equation}
where $q_1:=\frac{\alpha Nq_0}{Nq_0+\beta (N-4q_0)}$.

Indeed, for $ t >0$ set 
\[ \hbox{$\Omega^n_{1,t}:=\{ x \in \Omega: f(u_n(x)) \le (u_n^\beta (x)+1)^{\frac{t}{\beta}} \}$\quad  and \quad $ \Omega^n_{2,t} = \Omega \backslash \Omega^n_{1,t}$.}\]  
Since $1\leq q_0 \leq  \frac{N}{4}$ and $ \displaystyle \sup_n \|f(u_n)\|_{q_0} <+\infty$, we have  via the Sobolev embedding Theorem that 
$$\hbox{$\displaystyle \sup_n \|u_n\|_s<+\infty$  for every $s<\frac{Nq_0}{N-4q_0}$,  }$$
and hence on $ \Omega^n_{1,t}$ we have 
\begin{equation}\label{first}
\hbox{$ \displaystyle \sup_n \int_{\Omega^n_{1,t}}f^s(u_n) <+\infty$ for all $s<\frac{Nq_0}{t(N-4q_0)}$.}
\end{equation}
On $ \Omega^n_{2,t}$, we have $ f^{\alpha - \frac{\beta}{t}}(u_n) \le \frac{ f^\alpha(u_n)}{u_n^\beta+1 }, $
therefore  
\begin{equation}\label{2d}
\sup_n \int_{\Omega^n_{2,t}}f^{\alpha - \frac{\beta}{t}}(u_n)<+\infty.
\end{equation}
If $q_0< \frac{N}{4}$, then take $ t = \frac{Nq_0+\beta(N-4q_0)}{\alpha(N-4q_0)}$ in such a way that 
\[
\alpha - \frac{\beta}{t}=\frac{Nq_0}{t(N-4q_0)},
\]
 to obtain  that 
\[
 \hbox{$\displaystyle \sup_n  \|f(u_n)\|_s<+\infty$ for all $s<\frac{\alpha Nq_0}{Nq_0+\beta( N-4q_0)}$.}
 \]
 If $q_0= \frac{N}{4}$, then we can let $t\to +\infty$ in (\ref{2d}) and combine with (\ref{first}) to obtain that 
 \[
 \hbox{$\displaystyle \sup_n  \|f(u_n)\|_s <+\infty$ for all $s<\alpha$,}
 \]
 and then (\ref{iteration}) is proved. Note that for $1<q_0 \leq  \frac{N}{4}$ (\ref{iteration}) is equivalent to:
 \begin{equation}\label{iterationbis}
\hbox{if $ \displaystyle \sup_n \|f(u_n)\|_s <+\infty$  for every $ 1 \leq  s<q_0$,\,\,  then \,\,  $\displaystyle \sup_n \|f(u_n)\|_{s}<+\infty$ for every $s<q_1$, }
\end{equation}
where $q_1$ is as before.
 
By elliptic regularity theory, assumption (\ref{bound1}) implies for $q>\frac{N}{4}$ that $\displaystyle \sup_n \|u_n\|_\infty<+\infty$. When $1\leq q \leq  \frac{N}{4}$, by (\ref{iteration}) we can say that
$$\displaystyle \sup_n \|f(u_n)\|_s <+\infty \qquad \hbox{for every }s<q_1:=\frac{\alpha Nq_0}{Nq_0+\beta(N-4q_0)}.$$
If $q_1>\frac{N}{4}$ we are done. Otherwise, thanks to (\ref{iterationbis}) we can use an iteration argument to show that
$$\displaystyle \sup_n \|f(u_n)\|_s <+\infty \qquad \hbox{for every }s<q_{n+1}:=\frac{\alpha Nq_n}{Nq_n+\beta(N-4q_n)}$$
for every $n\geq 1$, as long as $q_n \leq \frac{N}{4}$. Since $\frac{(\alpha-\beta)N}{N-4\beta}>\frac{N}{4}$ when $\alpha>\frac{N}{4}$ and $1\leq q \leq  \frac{N}{4}$, an easy induction shows that the sequence $q_n$ is
\begin{itemize}
\item  increasing  to $\frac{(\alpha-\beta)N}{N-4\beta}$ when $\alpha \leq \frac{N}{4}$ and $1\leq q<\frac{(\alpha-\beta)N}{N-4\beta}$;
\item decreasing  to $\frac{(\alpha-\beta)N}{N-4\beta}$ when $\alpha \leq \frac{N}{4}$ and $ q>\frac{(\alpha-\beta)N}{N-4\beta}$;
\item  increasing and passes the value $\frac{N}{4}$ after a finite number of steps when $\alpha > \frac{N}{4}$.
\end{itemize}
Claims (1) and (2) are then established.
 \end{proof}
We can now deduce the following.
\begin{cor} \label{cor11} Suppose $(u_n)_n$ is a sequence of solutions of $(N_{\lambda_n})$ such that 
%Let either $f(t)=e^t$ or $f(t)=(t+1)^p$, $p>1$. Suppose that  
\begin{equation} \label{bound1111}
\sup_n  \int_\Omega f^q(u_n) <+\infty
\end{equation} 
for $q\geq 1$. Then $\displaystyle \sup_n \|u_n\|_\infty <+\infty$, in either one of the following two cases:
\begin{enumerate}
\item  $f(t)=e^t$ and $q \geq \frac{N}{4}$;   
\item $ f(t) = (t+1)^p$ and $ q > \frac{N}{4} \left( 1 - \frac{1}{p} \right) $.
% in the case . Then $\displaystyle \sup_n \|u_n\|_\infty <+\infty$.
\end{enumerate}
\end{cor}   
\begin{proof}
(1)\, For $q>\frac{N}{4}$ it follows by standard regularity theory. The case $q=\frac{N}{4}$ and $f(t)=e^t$  can be treated in the following way. Since $ e^{u_n}$ is uniformly bounded in $L^\frac{N}{4}(\Omega)$, by elliptic regularity theory and the Sobolev embedding Theorem $u_n$ is uniformly bounded in $W^{1,N}_0(\Omega)$.  The  Moser-Trudinger inequality  states that, for suitable $ \alpha >0$ and $ C_i >0$, there holds 
\[ \int_\Omega e^{\alpha |u |^\frac{N}{N-1}} dx \le C_0 + C_1 e^{ C_2 \| \nabla  u \|_{L^N}^N}, \qquad \forall u \in W_0^{1,N}(\Omega).\]
Now fix $  \tau > \frac{N}{4}$ and pick $ \tilde{C}$ big enough such that  
 \[ e^{ \tau z} \le \tilde{C} e^{ \alpha z^\frac{N}{N-1}},\] for all $ z \ge 0$.  Then we have 
 \[ \frac{1}{ \tilde{C}} \int_\Omega e^{ \tau  u_n} dx \le C_0 + C_1 e^{C_2 \| \nabla u_n \|_{L^N}^N} \le \bar{C},\] 
 and so we have $ e^{u_n}$ uniformly bounded in $ L^\tau(\Omega)$ for some $ \tau > \frac{N}{4}$. By elliptic estimates, the validity of (\ref{bound}) follows also in this case.\\

(2) \, The case where $f(t)=(t+1)^p$ and $\frac{N}{4} \left( 1 - \frac{1}{p} \right)<q\leq \frac{N}{4} $ follows from Theorem \ref{iterate} with the choice $\alpha=q+
%(1-\epsilon)
\frac{N}{4p}$, $\beta=\frac{N}{4}$, 
%(1-\epsilon)$, where 
since $\alpha> 
\frac{N}{4}$ and $\alpha>\beta$. 
Note that
$$\sup_n \int_\Omega \frac{f^\alpha(u_n)}{u_n^\beta+1}\leq C \sup_n \int_\Omega f^q(u_n)<+\infty,$$
for some $C>0$. 
\end{proof}
     
We now show that the standard assumption $\sup_n \|f(u_n)\|_q<+\infty$, $q>\frac{N}{4}$, which guarantees the uniform boundedness of $u_n$ can be weakened in a different way, through a uniform integrability condition on $f'(u_n)$. Indeed, we have the following result.

\begin{thm} \label{requires}  Suppose that  for some $q>\frac{N}{4}$ we have
\begin{equation} \label{bound3}
\sup_n  \int_\Omega f^s(u_n) <+\infty \qquad  \hbox{for every }1\leq s<\frac{N}{N-2}  
\end{equation} 
and 
\begin{equation} \label{bound4}
\sup_n \int_\Omega (f')^q (u_n) <+\infty.
\end{equation} 
Then, \begin{equation} \label{boundbound}
\displaystyle \sup_ n \|u_n\|_\infty<+\infty.
\end{equation}
\end{thm}   

\begin{proof}  Observe that $\tilde v_n=-\Delta u_n$ satisfies
$$\left\{  \begin{array}{ll}
\Delta^2 \tilde v_n \leq \lambda_n f'(u_n) \tilde v_n& \hbox{in }\Omega\\
\tilde v_n=0,\: -\Delta \tilde v_n=\lambda_n & \hbox{on }\partial \Omega.
\end{array} \right.$$
Introducing the function $w_n$ as the solution of
$$\left\{  \begin{array}{ll}
-\Delta w_n =\lambda_n & \hbox{in }\Omega\\
w_n=0 & \hbox{on }\partial \Omega,
\end{array} \right. $$
we are led to study uniform boundedness for $v_n=\tilde v_n-w_n$, a solution of
\begin{equation} \label{ineq}
\left\{  \begin{array}{ll}
\Delta^2  v_n \leq \lambda_n f'(u_n) v_n+\lambda_n f'(u_n) w_n& \hbox{in }\Omega\\
v_n=\Delta v_n=0 & \hbox{on }\partial \Omega.
\end{array} \right.
\end{equation}
Since $\lambda_n$ is bounded, by elliptic regularity theory we deduce that 
\begin{equation} \label{wn}
\sup_n \|w_n\|_\infty<+\infty,
\end{equation}
and then the uniform boundedness can be equivalently established on $\tilde v_n$ or $v_n$.  First we  show that under assumption (\ref{bound4}), the following hold: 
\begin{eqnarray}\label{iterationbisbis}
&&\hbox{if $ \displaystyle \sup_n \|v_n\|_s <+\infty$ $\forall\: 1 \leq  s <q_0$ and $q_0 \leq \frac{Nq}{4q-N}$, then $\displaystyle \sup_n  \|v_n \|_s<+\infty$ $\forall \: s<q_1$ }\\
&&\hbox{if $ \displaystyle \sup_n \|v_n\|_s <+\infty$  $\forall\: 1 \leq  s <q_0$ and $q_0> \frac{Nq}{4q-N}$, then $\displaystyle \sup_n  \|v_n \|_\infty<+\infty$,}  \label{iterationter}
\end{eqnarray}
where $q_1:=\frac{Nqq_0}{Nq_0+q(N-4q_0)}$. Indeed, by (\ref{bound4}) and (\ref{wn}) we get that
$$\lambda_n f'(u_n) v_n+\lambda_n f'(u_n) w_n \hbox{ uniformly bounded in }L^s(\Omega), \: \forall \:1\leq s<\frac{qq_0}{q+q_0} .$$ 
Thanks to (\ref{ineq}), by elliptic regularity theory and the maximum principle the previous estimate translates into: if $\frac{qq_0}{q+q_0}\leq \frac{N}{4}$
$$\sup_n \|v_n\|_s<+\infty \qquad \hbox{for every } 1\leq s <q_1,$$
and if $\frac{qq_0}{q+q_0}> \frac{N}{4}$
$$\sup_n \|v_n\|_\infty<+\infty .$$
Therefore, (\ref{iterationbisbis})-(\ref{iterationter}) are established.

\medskip \noindent Thanks to (\ref{wn}), by elliptic regularity theory assumption (\ref{bound3}) reads on $v_n$ as $\displaystyle \sup_n \|v_n\|_\infty<+\infty$ if $N=2,3$ and
\begin{equation} \label{bound3bis}
\sup_n  \int_\Omega (v_n)^s <+\infty \qquad  \hbox{for every }1\leq s<\frac{N}{N-4} 
\end{equation} 
if $N\geq 4$. For $N\geq 4$, set $q_0=\frac{N}{N-4} $ and inductively $q_{i+1}= \frac{Nqq_i}{Nq_i+q(N-4q_i)}$ as long as $q_i \leq \frac{Nq}{4q-N}$ so to get
$$\sup_n  \|v_n \|_s<+\infty \qquad 	\hbox{for every }1\leq s<q_{n+1},$$
in view of (\ref{iterationbis}). Since $q> \frac{N}{4}$, the sequence $q_i$ is increasing and passes $ \frac{Nq}{4q-N}$ after a finite number of steps. As soon as $q_i$ becomes larger than $ \frac{Nq}{4q-N}$, 
we can use (\ref{wn}) and (\ref{iterationter}) to get an uniform $L^\infty-$bound on $-\Delta u_n=v_n+w_n$, and in turn the validity of (\ref{boundbound}) follows by elliptic estimates.
 \end{proof}

\section{A general regularity result for low dimensions}

To the best of our knowledge the only available  energy estimates for smooth, semi-stable solutions $u$ of $(N_\lambda)$ so far, are given by 
\begin{equation}
 \int_\Omega f'(u) u^2 dx \le \int_\Omega f(u) u dx.
 \end{equation}
 To see this, take $ \psi = u$ in (\ref{stable}) and integrate by parts $(N_\lambda)$ against $u$, and then equate.    
 In view of Corollary \ref{cor11},  this yields the following 
 \begin{enumerate}
 \item If $f(t)=e^t$, then $ e^{u^*} (u^*)^2 \in L^1(\Omega)$ and $u^*$ is then regular  for $N\leq 4$.
 
 \item If $f(t) = (t+1)^p$, then  $ (u^*+1)^p \in L^\frac{p+1}{p}(\Omega)$, therefore  $u^*$ is a regular solution for $2\leq N< \frac{4(p+1)}{p-1}$ (equivalently $N \leq 4$ or $1\leq p<\frac{N+4}{N-4}$ and $N> 4$)
 
 \item  If $ f(t)= (1-t)^{-2}$, then   $ (1-u^*)^{-2} \in L^\frac{3}{2}(\Omega)$  and $u^*$ is regular for $N\leq 4$. See Chapter 12 of \cite{memsbook}. 
 \end{enumerate}
%b) Regularity results on $u^*$ translate into regularity properties for any weak semi-stable solution. Indeed, by points (2)-(4) above we see that a weak semi-stable solution coincides necessarily either with the classical solution  $u_\lambda$ or with $u^*$. Our preference for not stating the results in this generality is to avoid the technical details of defining precisely what we mean by a suitable weak solution. 
%\begin{remark} \label{caff}
%For exponential and power nonlinearities, by Remark \ref{rk11}  and Corollary \ref{cor11}  we obtain the basic regularity result for $u^*$ in terms of $N$: $2\leq N\leq 4$ for $f(t)=e^t$ and $2\leq N<4 \frac{p+1}{p-1}$ (equivalently $N=2,3,4$ or $1\leq p<\frac{N+4}{N-4}$ and $N> 4$) for  $f(t)=(t+1)^p$.
%\end{remark}

We shall substantially improve on these results in the next sections. For now, we start by considering the case of a general superlinear $f$ and establish  a fourth order  analogue of the  results   of Nedev \cite{Nedev} for $N\leq 3$ and Cabre \cite{Cabre} for $N=4$,  regarding the regularity of the extremal solution of  second order eigenvalue problems with a nonlinearity of type (R).   

\begin{thm}\label{Nedev} Let $f$ be a nonlinearity of type (R). Then the extremal solution $u^*$ of $(N_\lambda)$ is regular for $N\leq 5$, while $ f(u^*) \in L^q(\Omega)$ for all $ q < \frac{N}{N-2}$ if $N\geq 6$. 
\end{thm} 
We shall split the proof in several lemmas that may have their own interest, in particular   
 the simple new energy estimate given in Lemma \ref{main-energy} below,  coupled with a pointwise estimates on $ -\Delta u $ given in Lemma \ref{max}, and which was motivated by the proof of Souplet  of the Lane-Emden conjecture in four space dimensions \cite{soup}. 
We start by the latter (next two lemmas) which does not require the stability of  the solutions. 

\begin{lemma} \label{max} Suppose $ u$ is a solution of $ (N_\lambda)$ and $g$ is a smooth function defined on the range of $u$ with $ f(t) \ge g(t) g'(t)$ and $ g(t), g'(t), g''(t) \ge 0$ on the range of $ u$ with $ g(0)=0$.   Then 
\begin{equation}-\Delta u \ge \sqrt \lambda g(u) \qquad \mbox{ in $ \Omega$}.
\end{equation}
\end{lemma}  
\begin{proof} Define $ v:=-\Delta u - \sqrt  \lambda g(u)$ and so $ v =0$ on $ \pOm$ and a computation shows that 
\[ -\Delta v + \sqrt  \lambda g'(u) v = \lambda [f(u) -g(u) g'(u)] + \sqrt  \lambda g''(u) | \nabla u|^2 \qquad \mbox{ in $ \Omega$}.\]   The assumptions on $ g$ allow one to apply the maximum principle and obtain that $ v \ge 0$ in $ \Omega$. 
\end{proof}

Now we use the stability condition on the solution.
\begin{lemma} \label{main-energy} Suppose $ u$ is a semi-stable solution of $ (N_\lambda)$.  Then 
\begin{equation}
\int_\Omega f''(u) (-\Delta u)  | \nabla u|^2 dx \le \lambda \int_\Omega f(u) dx.
\end{equation}

\end{lemma}

\begin{proof} Set $ \psi=\Delta u$ in (\ref{stable}) to arrive at 
\[ I:= \int_\Omega f'(u) (\Delta u)^2 dx \le \int_\Omega \Delta^2 u f(u) dx =:J.\]  Now an integration by parts shows that 
\begin{eqnarray*}
 I&=& \int_\Omega f''(u) (-\Delta u) | \nabla u|^2 dx - \int_\Omega f'(u) \nabla u \cdot \nabla \Delta u dx\\ 
J&=& \lambda \int_\Omega f(u) dx - \int_\Omega f'(u) \nabla u \cdot \nabla \Delta u dx,
\end{eqnarray*}
in view of $f(0)=1$. Since $ I \le J$ one obtains the result. 
\end{proof} 

\begin{lemma}  \label{ttt}  Suppose $u$ is a semi-stable solution of $(N_\lambda)$ and that $ g$ satisfies the conditions of Lemma \ref{max}.  If  
$H(u):= \int_0^u f''( \tau) g(\tau) d \tau$, then 
\begin{equation}
\int_\Omega g(u) H(u) dx \le  \int_\Omega f(u) dx.
\end{equation}
\end{lemma} 
\begin{proof}    We rewrite the result from Lemma \ref{main-energy} as 
\begin{eqnarray*}
\lambda \int_\Omega g(u) H(u)dx & \le & \sqrt \lambda  \int_\Omega (-\Delta u) H(u) dx 
=\sqrt \lambda \int_\Omega \nabla H(u) \cdot \nabla u dx \\
&=& \sqrt \lambda \int_\Omega H'(u) | \nabla u|^2 dx 
 \le \lambda  \int_\Omega f(u) dx,
\end{eqnarray*} 
where the two inequalities use the pointwise bound from Lemma \ref{max}. 
\end{proof}

\begin{lemma} \label{crucial} Suppose $ u\geq 0$ is a semi-stable solution of $ (N_\lambda)$.  Then  
\begin{equation} \label{init}
\hbox{$\int_\Omega \frac{ f(u)^\frac{3}{2} }{\sqrt{u}+1} dx \le C$\quad  and  \qquad $\int_\Omega f(u) dx \le C,$}
\end{equation} 
for some constant $C>0$ independent of $ \lambda$ and $u$. 
\end{lemma} 

\begin{proof} Define   for $ u \ge 0$, the function
\[
g(u):= \sqrt{2} \left( \int_0^u (f(t)-1) dt \right)^\frac{1}{2}.
\] 
Clearly $ g(0) =0$ and $ g \ge 0$.    Now square $g$ and take a derivative to see that $ 2 g(u) g'(u) =2 (f(u)-1)$ and so we satisfy the requirement that $ f(u) \ge g(u) g'(u)$.  Also from this we see that $ g'(u) \ge 0$.  \\
  We now show that $ g''(u) \ge 0$.  Note that $ g''(u)$ has the same sign as 
  \[ \gamma(u) := f'(u) \int_0^u (f(t)-1)dt - \frac{1}{2} ( f(u) -1)^2.\]  Now $ \gamma(0)=0$ and 
  \[ \gamma'(u) = f'(u) (f(u)-1) + f''(u) \int_0^u (f(t)-1) dt - (f(u)-1) f'(u) ,\] and so $ \gamma'(u) \ge 0$ and hence $ \gamma(u) \ge 0$.   \\
  By Lemma \ref{ttt} we have 
 \begin{eqnarray} \label{cruc} 
 \int_\Omega g(u) H(u) dx \le \int_\Omega f(u) dx, 
 \end{eqnarray}
 where 
  \[ H(u):= \int_0^u f''( \tau) g(\tau) d \tau.\]    Without loss of generality, we can assume that $\int_0^1 (f(t)-1) dt>0$.  For $ u >1$ we have 
  \begin{eqnarray*}
  H(u)  \ge  \sqrt{ 2} \int_1^u f''( \tau)  \left( \int_0^\tau (f(t)-1) dt \right)^\frac{1}{2} d \tau \geq  \sqrt 2C_0 ( f'(u) - f'(1))
  \end{eqnarray*} where $ C_0:= \left( \int_0^1 (f(t)-1) dt \right)^\frac{1}{2}$. Since by convexity $ f'(u) \ge \frac{f(u)-1}{u} \rightarrow \infty$ as $ u \rightarrow \infty$, we can find $M>0$ large so that
  $$H(u)\geq C_0 f'(u) 
   \qquad \forall \:u\geq M.$$
Since
 \begin{eqnarray} \label{cructer} 
\int_0^u (f(t)-1)\geq \int_1^u (f(t)-1)\geq (1-\frac{1}{f(1)}) \int_1^u f(t)dt\geq (1-\frac{1}{f(1)}) \int_0^u f(t)dt
\end{eqnarray}
for $u\geq 1$, from (\ref{cruc}) and the above estimate we see that 
\begin{eqnarray*} 
&&\int_\Omega f'(u)   \left( \int_0^{ u(x)} f(t) dt  \right)^\frac{1}{2} dx \\
&&\le \int_{\{ u\geq M\} } f'(u)   \left( \int_0^{ u(x)} f(t) dt  \right)^\frac{1}{2} dx+ f'(M) \left( \int_0^M f(t) dt  \right)^\frac{1}{2} |\Omega|\\
&& \leq C_0^{-1}(1-\frac{1}{f(1)})^{-1} \int_\Omega H(u)   \left( \int_0^{ u(x)} (f(t)-1 ) dt  \right)^\frac{1}{2} dx+ f'(M) \left( \int_0^M f(t) dt  \right)^\frac{1}{2} |\Omega|\\
&& \leq C_1 \int_\Omega f(u) dx
\end{eqnarray*}  
for some $C_1>0$ (independent of $\lambda$ and $u$), in view of $|\Omega|\leq \int_\Omega f(u)$.
Defining 
\[ 
h(u):= u (f')^2(u) \int_0^u f(t) dt - \frac{1}{6} ( f^\frac{3}{2}(u) -1)^2,
\] 
we have that $ h \ge 0$.  Note $ h(0)=0$ and 
\[ h'(u) = 2u f'(u) f''(u) \int_0^u f(t) dt + u (f')^2(u) f(u) + I,\]  where
\[ I = (f')^2(u) \int_0^u f(t) dt - \frac{1}{2}f^2(u) f'(u)  + \frac{1}{2} f^\frac{1}{2}(u) f'(u).
\] 
Since 
\[ 
f'(u)^2 \int_0^u f(t) dt  \ge f'(u) \int_0^u f'(t) f(t) dt = f'(u) \frac{f(u)^2}{2} - \frac{f'(u)}{2},
\]  we have that $ I \ge 0$, and then $ h'(u) \ge 0$.  Hence, $h(u)\geq 0$ leads to the fundamental estimate:
 \begin{eqnarray} \label{fund} 
 f'(u) \left( \int_0^u f(t) dt \right)^{\frac{1}{2}}  \geq \frac{f^\frac{3}{2}(u) -1}{\sqrt 6 (\sqrt u+1)}  \qquad \forall\: u\geq 0.
 \end{eqnarray}
So by (\ref{fund}) we get that 
\[ \int_\Omega \frac{ f(u)^\frac{3}{2} }{\sqrt{u}+1} dx \leq \int_\Omega \frac{ f(u)^\frac{3}{2} -1}{\sqrt{u}+1} dx+|\Omega| 
\le (\sqrt 6 C_1+1) \int_\Omega f(u) dx.\] 
Since $ f$ is superlinear at $ \infty$ this implies the validity of  (\ref{init}). For later purposes, note that from the above estimates there holds
$$g(u) H(u)\geq \sqrt 2 C_0 (1-\frac{1}{f(1)}) ^{\frac{1}{2}} \left(\int_0^u f(t)dt\right)^{\frac{1}{2}} f'(u) \geq \sqrt 2 C_0 (1-\frac{1}{f(1)}) ^{\frac{1}{2}} \frac{f^\frac{3}{2}(u) -1}{\sqrt 6 (\sqrt u+1)} $$
for $u \geq M$, and then by superlinearity of $f$ at $\infty$
$$\frac{g(u) H(u)}{f(u)}\to +\infty\qquad \hbox{as }u\to +\infty.$$ 
Hence, for general nonlinearities $f$ of type (R) we can re-state Lemma \ref{ttt} as
\begin{equation} \label{tttbis}
\int_\Omega g(u) H(u) \leq C
\end{equation}
for every semi-stable solution $u$ of $(N_\lambda)$, where $g(u)$ is exactly as before and $C$ is independent of $\lambda$ and $u$.
\end{proof}

{\bf Proof of Theorem \ref{Nedev}:} Recalling that $u^*$ is the limit of the classical solutions $u_\lambda$ as $\lambda \nearrow \lambda^*$, it follows immediately from estimate (\ref{init}) and Theorem  \ref{iterate}. Indeed, in this case we can take $q=1$, $\alpha =\frac{3}{2}$ and $\beta=\frac{1}{2}$ to conclude that $u^* \in L^\infty(\Omega)$ when $N\leq 5$, and $f(u^*)\in L^q(\Omega)$ for every $q< \frac{N}{N-2}$ when $N\geq 6$.

\section{Regularity in higher dimension (I)} 
In this section, we will consider nonlinearities $f$ of type (R) which satisfy the following  growth condition:
\begin{equation}  \label{ugammabasso}
 \liminf\limits_{t \to +\infty}\frac{f (t)f'' (t)}{(f')^2(t)}>0.
\end{equation}  
The aim is to gain dimensions $N=6,7$ in Theorem \ref{Nedev} by showing that $L^2-$bounds on $f(u)$ are still in order, as we state in the following theorem.
\begin{thm} \label{L2}  Let $f$ be a nonlinearity of type (R) so that (\ref{ugammabasso}) holds. Let $u$ be a semi-stable solution of $(N_\lambda)$.  Then 
\[ \int_\Omega f^2(u) dx \le C, \]
where $C>0$ is independent of $\lambda$ and $u$.
\end{thm}  
By standard elliptic regularity theory, Theorem \ref{L2} immediately yields the following improvement on Theorem \ref{Nedev}.
\begin{thm}\label{Nedevimproved} Let $f$ be a nonlinearity of type (R) so that (\ref{ugammabasso}) holds. Then the extremal solution $u^*$ of $(N_\lambda)$ is regular for $N\leq 7$.
\end{thm}
\textbf{Proof (of Theorem \ref{L2})} Since
$$\left( f'(t) \int_0^t f \right)'=f''(t) \int_0^t f+f'(t)f(t)\geq f'(t)f(t)=\left(\frac{1}{2} f^2(t) \right)',$$
one can integrate on $[0,u]$ to get
$$f'(u)\int_0^u f \geq \frac{1}{2}f^2(u)-\frac{1}{2}.$$ 
Since $f(u)\to +\infty$ as $u \to +\infty$, we can find $M\geq 1$ large so that
\begin{equation} \label{n1}
f'(u)\int_0^u f \geq \frac{1}{4}f^2(u)\qquad \forall\: u\geq M.
\end{equation}
Setting 
\begin{equation*} % \label{ugammabasso}
\delta:= \liminf\limits_{t \to +\infty}\frac{f (t)f'' (t)}{(f')^2(t)}>0, 
\end{equation*}  
we can --modulo taking a larger $M$ --
%larger, by (\ref{ugammabasso}) we can 
also assume that 
\begin{equation} \label{n2}
f (u)f'' (u) \geq \frac{\delta}{2} (f')^2(u) \qquad \forall\: u\geq M.
\end{equation}
By (\ref{n1})-(\ref{n2}) we get
$$ \left[ \left(\int_1^u f''(t) (\int_0^t f)^{\frac{1}{2}}\right)  (\int_0^u f)^{\frac{1}{2}}\right]'  \geq f''(u) \int_0^u f\geq \frac{\delta}{2}  \frac{(f')^2(u)}{f(u)} \int_0^u f
\geq \frac{\delta}{8} f(u) f'(u)=\frac{\delta}{16}( f^2(u))'$$
for all $u \geq M$, which, integrated once more in $[M,u]$, $u\geq M$, yields to 
$$  \left(\int_1^u f''(t) (\int_0^t f)^{\frac{1}{2}}\right)  (\int_0^u f)^{\frac{1}{2}} \geq \frac{\delta}{16} f^2(u)-\frac{\delta}{16}f^2(M).$$
Then we can find $N\geq M$ large so that for $u\geq N$ we have
$$  \left(\int_1^u f''(t) (\int_0^t f)^{\frac{1}{2}}\right)  (\int_0^u f)^{\frac{1}{2}} \geq \frac{\delta}{32} f^2(u).$$
Setting as always $g(u)=\sqrt 2 \left( \int_0^u (f(t)-1)dt \right)^{\frac{1}{2}} $,  by (\ref{cructer}) we can now deduce
$$g(u)H(u)\geq 2 \left(\int_1^u f''(t) (\int_0^t (f-1)ds)^{\frac{1}{2}}\right)  (\int_0^u (f-1)dt )^{\frac{1}{2}}
\geq  \frac{\delta}{16}(1-\frac{1}{f(1)}) f^2(u)$$
for $u \geq N$. By Lemma \ref{ttt} as re-stated in (\ref{tttbis}) we finally get that
$$\int_\Omega f^2(u) dx \leq \int_{\{u\geq N\}} f^2(u)dx+f^2(N)|\Omega| \leq   \frac{16}{\delta}(1-\frac{1}{f(1)})^{-1} \int_\Omega g(u)H(u)+f^2(N)|\Omega|\leq C$$
for some $C>0$ independent of $\lambda$ and $u$.\\

Theorem \ref{L2} combined with Corollary \ref{cor11} gives also immediately the following results. 

\begin{cor} \label{exp+up}
When $f(t)=(t+1)^p$, $p>1$, the extremal  solution $u^*$ of $(N_\lambda)$ is regular if either $N\leq 8$ or if $N\geq 9$ and $p<\frac {N}{N-8}$. When $f(t)=e^t$, this is true for $N\leq 8$.
\end{cor}
%As a by-product, we now have\\
%\textbf{Proof (of Corollary \ref{exp+up})} 
%\begin{proof} When $f(t)=(t+1)^p$, it turns out that $\gamma=1-\frac{1}{p}$, and then $u^*$ is a regular solution whenever $N\leq 5$ and $6\leq N <8+\frac{8}{p-1}$. We can collect the two cases as $N\leq 8$ or $N\geq 9$ and $p<\frac{N}{N-8}$. When $f(t)=e^t$, we have that $\gamma=1$, and then $u^*$ is a regular solution for $N<8$. The missing dimension $N=8$ follows directly from Theorem \ref{requires} in view of the identity, 
%$$\int_\Omega (f')^{\frac{2}{\gamma}}(u) dx=\int_\Omega (f')^2(u) dx=\int_\Omega f^2(u) dx=\int_\Omega f^{\frac{N}{4}}(u) dx.$$
%\end{proof}
 
\section{Regularity in higher dimension (II)} 
We are still considering nonlinearities $f$ of type (R). For $N\geq 6$ we want to improve upon Theorem \ref{Nedev} under the following  growth condition on $f$
\begin{equation}  \label{ugamma}
\gamma:= \limsup\limits_{t \to +\infty}\frac{f (t)f'' (t)}{(f')^2(t)}<+\infty.
\end{equation}  
Typical examples of such nonlinearities are again $f(t)=e^t$ (with $\gamma=1$) and $f(t)=(t+1)^p$ (with $\gamma=1-\frac{1}{p}$). The aim is to get the regularity of the extremal solution also in dimensions higher than $5$ for values of $\gamma$ not too large: 
\begin{thm} Let $N\geq 6$ and $f$ be a nonlinearity of type (R) satisfying (\ref{ugamma}). The extremal  solution $u^*$ of $(N_\lambda)$ is regular for $N<\frac{8}{\gamma}$.
\label{nera} \end{thm} 
The validity of Theorem \ref{nera} follows easily Theorem \ref{Nedev}, Theorem \ref{requires} and the following crucial estimate for stable solutions. To apply Theorem \ref{requires}, we need to require exactly $\gamma<\frac{8}{N}$ when $N\geq 6$, and (\ref{ugamma}) guarantees the validity of (\ref{ugammadef}) with $0<\gamma+\epsilon<2$ and $M>0$ large enough.

\begin{thm} \label{gamma.estimate} Let $f$ be a nonlinearity of type (R) so that
\begin{equation}  \label{ugammadef}
f (u)f'' (u) \leq \gamma (f')^2(u)\qquad \forall \: u\geq M,
\end{equation}  
for some $0<\gamma<2$ and $M>0$. Let $u$ be a  semi-stable solution of $(N_\lambda)$. Then
\begin{equation}  \label{g.estimate}
\int_\Omega (f')^{ \frac{2}{\gamma} }(u)dx \le C,
\end{equation}
where $C>0$ is a constant independent of $u$ and $\lambda$.
\end{thm}   
\begin{proof} Re-write (\ref{ugammadef}) as 
\[ \frac{d}{dt} \log( f'(t)) \le \frac{d}{dt} \log( f^\gamma(t)) \qquad \forall\: t \geq M\] 
and integrate over $[M,u]$ to deduce that
\begin{equation} \label{ddt}
f'(u) \le \frac{ f'(M)}{f^\gamma(M)} f^\gamma(u)\qquad \forall \: u \geq M. 
\end{equation}
Since $f(u)\geq f(0)=1$, we can write that
\begin{equation}  \label{scii} 
f'(u) \le C_0 f^\gamma(u)\qquad \forall \: u \geq 0,
\end{equation} 
where $C_0$ is a suitable large constant. Setting 
\[ \Gamma(u):= \int_0^u f(t) dt - \frac{1}{(2- \gamma) C_0} \left( f^{2- \gamma}(u) -1 \right),\] 
one notes that $ \Gamma(0)=0$ and 
$ \Gamma'(u) = f(u) - \frac{ f^{1- \gamma} (u)f'(u)}{C_0} \ge 0.$
Hence, the  following estimate holds for every  $u\geq 0$:  
\begin{equation} \label{vv}
\sqrt{ \int_0^u f(t) dt } \ge \frac{1}{\sqrt{ (2- \gamma) C_0}} \left( f^{2- \gamma}(u)-1 \right)^\frac{1}{2}.
\end{equation}
As in the previous section, set $g(u) = \sqrt{2 } \left( \int_0^u ( f(t)-1) dt \right)^\frac{1}{2}$ in such a way that it satisfies the assumptions of Lemma \ref{ttt}. 
By  (\ref{cructer}), (\ref{vv}) and the superlinearity of $f$ at $\infty$, we can find $N\geq 1$ large so that for all $u\geq N$, 
\begin{eqnarray*}
g(u)&\geq& \sqrt 2 (1-\frac{1}{f(1)})^\frac{1}{2} \left(\int_0^u f(t)dt \right)^\frac{1}{2}
\geq \sqrt 2 \left(\frac{f(1)-1}{(2-\gamma)C_0 f(1)} \right)^\frac{1}{2} (f^{2-\gamma}(u)-1)^\frac{1}{2}\\
&\geq & \left(\frac{f(1)-1}{(2-\gamma)C_0f(1)} \right)^\frac{1}{2} f^{1-\frac{\gamma}{2}}(u). 
\end{eqnarray*}
Setting $C_1:= \left(\frac{f(1)-1}{(2-\gamma)C_0f(1)} \right)^\frac{1}{2}$,  we use (\ref{scii}) to find $N'\geq N$ sufficiently large so that 
for $ u \ge N'$ 
\begin{eqnarray*}
H(u) & := & \int_0^u f''(t)g(t) dt    
   \ge  C_1 \int_N^u f''(t) f^{ 1 - \frac{ \gamma}{2}} (t) dt 
   \ge  C_1 C_0^{\frac{1}{2}-\frac{1}{\gamma}}  \int_N^u  f''(t) (f')^{\frac{1}{\gamma}- \frac{1}{2}}(t)  dt \\
  & =& C_1 C_0^{\frac{1}{2}-\frac{1}{\gamma}} \frac{2\gamma}{\gamma+2} \left( (f')^{\frac{1}{\gamma} + \frac{1}{2}}(u) - (f')^{\frac{1}{\gamma} + \frac{1}{2}} (N)\right)
  \geq C_1 C_0^{\frac{1}{2}-\frac{1}{\gamma}} \frac{\gamma}{\gamma+2}  (f')^{\frac{1}{\gamma} + \frac{1}{2}}(u), 
 \end{eqnarray*}
where we have used the convexity of $f$ and the fact that  $f'(u)\to +\infty$ as $u \to +\infty$. In conclusion, setting 
$$C_2:= C_1 C_0^{\frac{1}{2}-\frac{1}{\gamma}} \frac{\gamma}{\gamma+2}\left(\frac{f(1)-1}{(2-\gamma)C_0f(1)} \right)^\frac{1}{2}$$ 
we have that
$$\int_\Omega (f')^{\frac{1}{\gamma} + \frac{1}{2}}(u) f^{1-\frac{\gamma}{2}}(u) \leq C_2^{-1} \int_\Omega H(u) g(u) +(f')^{\frac{1}{\gamma} + \frac{1}{2}}(N') f^{1-\frac{\gamma}{2}}(N') |\Omega|.$$
To complete the proof of Theorem \ref{gamma.estimate}, it suffices to couple this lower bound with (\ref{tttbis}) to obtain
$$\int_\Omega (f')^{\frac{1}{\gamma} + \frac{1}{2}}(u) f^{1-\frac{\gamma}{2}}(u) \leq C,$$
and then by (\ref{ddt}) to get
$$\int_\Omega (f')^{\frac{2}{\gamma} }(u) \leq C'$$
for any stable solution $u$, where $C,C'$ are independent of $\lambda$ and $u$.
\end{proof} 
\begin{remark} \rm a) As a by-product of the above theorem,  one obtains again the improvements  for the exponential and power nonlinearities established in Corollary \ref{exp+up}. 
%the following improvement on the results mentioned at the beginning of section 3. 
%(compare with Remark \ref{caff}).
%\begin{cor} \label{exp+up}
%When $f(t)=(t+1)^p$, $p>1$, the extremal  solution $u^*$ of $(N_\lambda)$ is regular if either $N\leq 8$ or if $N\geq 9$ and $p<\frac {N}{N-8}$. When $f(t)=e^t$, this is true for $N\leq 8$.
%\end{cor}
%As a by-product, we now have\\
%\textbf{Proof (of Corollary \ref{exp+up})} 
%\begin{proof} 
Indeed, when $f(t)=(t+1)^p$, it turns out that $\gamma=1-\frac{1}{p}$, and then $u^*$ is a regular solution whenever $N\leq 5$ and $6\leq N <8+\frac{8}{p-1}$. We can collect the two cases as $N\leq 8$ or $N\geq 9$ and $p<\frac{N}{N-8}$. When $f(t)=e^t$, we have that $\gamma=1$, and then $u^*$ is a regular solution for $N<8$. The missing dimension $N=8$ follows directly from Theorem \ref{requires} in view of the identity, 
$$\int_\Omega (f')^{\frac{2}{\gamma}}(u) dx=\int_\Omega (f')^2(u) dx=\int_\Omega f^2(u) dx=\int_\Omega f^{\frac{N}{4}}(u) dx.$$
%\end{remark}
%\end{proof}
%\begin{remark} \rm   
b) Recall from \cite{DDGM} that when $ f(t) = e^t$ on the unit ball in $ \IR^N$, the extremal solution associated with $(D_\lambda)$ is then smooth provided $ N \le 12 $.     This suggests that our regularity results are not optimal, which likely is a product of the deficient energy estimate obtained in Lemma \ref{main-energy} above.\\
c) Integrating once more (\ref{ddt}) one sees that
\begin{itemize} 
\item $f(u)\leq f(M)e^{\frac{f'(M)}{f(M)}(u-M)}$ for $u\geq M$ when $\gamma=1$
\item $f(u)\leq \left[f^{1-\gamma}(M)+(1-\gamma)\frac{f'(M)}{f^\gamma(M)}(u-M)\right]^{\frac{1}{1-\gamma}}$ for $u \geq M$ when $0<\gamma<1$.
\end{itemize}
This explains why (\ref{ugamma}) is sometimes referred to as a growth condition for $f$.\\
d) For exponential and power nonlinearities, the uniform bound (\ref{g.estimate}) can be re-formulated as an $L^2-$bound on $f(u)$. Indeed, one can define $g(u)$ as above, and use Lemma \ref{ttt} to deduce directly  such a bound, which in turn shows that the loss in optimality is not really coming from Theorem  \ref{gamma.estimate}.
\end{remark}

\section{Singular  nonlinearities}
Nonlinearities of the form $(1-u)^{-p}$, $p >0$, have recently attracted much attention, due to their connection with the so-called MEMS (Micro Electro-Mechanical Systems) technology. Neglecting torsion effects, one is led to study the second-order nonlinear eigenvalue problem  with $p=2$. In this case, the picture is well understood \cite{EGG, GG} and we refer the interested reader to the recent monograph \cite{memsbook}. The fourth-order case has been firstly addressed in \cite{LY} both in the form $(D_\lambda)$ and $(N_\lambda)$. As already mentioned in the introduction, for problem $(D_\lambda)$ the existence of the minimal branch on the unit ball has been proved in \cite{CDG} along with its compactness for $N\leq 8$ \cite{CEG}. Since a maximum principle holds for $(N_\lambda) $, the existence of the minimal branch for $(N_\lambda)$ on a general domain follows by the same argument as in \cite{CDG} for $(D_\lambda)$.\\
We will consider now the question of regularity of $u^*=\displaystyle \lim_{\lambda \nearrow \lambda^*}u_\lambda$ in terms of the dimension $N$. We will not consider general nonlinearities of type (S) in the sequel, but we shall restrict our attention to the interesting case $(1-u)^{-p}$ . The general case has been addressed in \cite{CES} for the second-order case, and growth conditions as (\ref{ugammabasso}) and (\ref{ugamma}) are no longer sufficient for the analysis.  
%To avoid further technicalities, we will just consider the purely singular nonlinearity $(1-u)^{-2}$ and  
We now establish the following result. 

\begin{thm} \label{sing} Suppose $  p >1$ and $ p \neq 3$. Then the extremal solution $ u^*$ is regular (i.e. $ \sup_\Omega u^* <1$) provided  $  N \le \frac{8p}{p+1}$. 

\end{thm}

This will follow immediately from the following two theorems.

\begin{thm} Let $ u_n$ denote a sequence of solutions of $(N_{\lambda_n})$ such that there is some $ \alpha > 1$ and $\alpha \ge \frac{(p+1)N}{4p}$ such that $\sup_n \| f(u_n)\|_\alpha < \infty$. 
Then $ \sup_n \| u_n\|_\infty<1$.  

\end{thm}

\begin{proof} We suppose that $ N$ is big enough so that $ \frac{(p+1)N}{4p} >1$,  the lower dimensional cases being similar we omit their details.   If $ f(u_n)$ is bounded in $ L^\frac{(p+1)N}{4p}$, then by elliptic regularity we have $ u_n$ bounded  in $ W^{4,\frac{(p+1)N}{4p}}$.   By the Sobolev imbedding theorem we have  $ u_n$ bounded in the space $ C^{ 4 - \left[ \frac{4p}{p+1} \right] -1,  \left[ \frac{4p}{p+1} \right] +1 - \frac{4p}{p+1}}( \overline{\Omega})$.
% where $ [\cdot ]$ denotes the floor function.   
This naturally breaks into the two cases:  
 \begin{itemize}
\item $1<p<3$ and then  $u_n$ is bounded in $C^{1, \frac{3-p}{p+1}}$ 
\item $p>3$ and $ u_n$ is then bounded in $ C^{0,\frac{4}{p+1}}$.  
\end{itemize}
We now let $ x_n \in \Omega$ be such that $ u_n(x_n) = \max_\Omega u_n$.  We claim that there exists some $ C>0$, independent of $n$, such that 
\[ | u_n(x) - u_n(x_n)| \le C |x-x_n|^\frac{4}{p+1}, \qquad x \in \Omega.\] For the second case this is immediate, while for the first we use the fact that $ \nabla u_n(x_n)=0$ and the fact that there is some $ 0 \le t_n \le 1$ such that 
\begin{eqnarray*}
u_n(x) - u_n(x_n) &=& \nabla u_n( x_n+t_n(x-x_n)) \cdot  (x-x_n) \\
&=& \left(  \nabla u_n( x_n+t_n(x-x_n)) - \nabla u_n(x_n) \right) \cdot  (x-x_n)
\end{eqnarray*} along with the fact that $ \nabla u_n$ is bounded in $ C^{0,\frac{3-p}{p+1}}$ to show the claim.  

To complete the proof, we work towards a contradiction, and assume,  after passing to a subsequence,  that $ u_n(x_n) = 1-\E_n \rightarrow 1$.   By passing to another subsequence, we can assume that $ u_n$ converges in $ C( \overline{\Omega})$ which along with the boundary conditions guarantees that $ x_n \rightarrow x_0 \in \Omega$.    Then one has 
\begin{eqnarray*}
1-u_n(x) & =& 1-u_n(x_n) + u_n(x_n) - u_n(x) \\
&=& \E_n + u_n(x_n) - u_n(x) \\
& \le & \E_n + C |x-x_n|^\frac{4}{p+1}, 
\end{eqnarray*} and so there is some $ C_p>0$ such that  
\[ \left( 1-u_n(x) \right)^\frac{(p+1)N}{4} \le C_p \left( \E_n^\frac{(p+1)N}{4} + |x-x_n|^N \right).\] From this one sees that 
\[ f(u_n(x))^\frac{(p+1)N}{4p} \ge \frac{C_p^{-1}}{\E_n^\frac{(p+1)N}{4} + |x-x_n|^N}:=h_n(x).\]   But since $ x_n \rightarrow x_0 \in \Omega$ and $ \E_n \rightarrow 0$,  ones sees that $ \int_\Omega h_n(x) dx \rightarrow \infty$  which contradicts the integrability condition on $ f(u_n)$.  Hence we must have $ \sup_n \| u_n \|_\infty <1$. 

\end{proof}

We now obtain the familiar $L^2$ bound on $ f(u)$ for semi-stable solutions.  We prefer to prove this results using an explicit calculation, even if this result also follows from Theorem 4.1.  

\begin{thm} Suppose $ p>1$ and  $u \ge 0$ is a semi-stable solution of $(N_\lambda)$.  Then 
\[  \| f(u) \|_2 \le C,\] where $ C$ is independent of $ u$ and $ \lambda$. 

\end{thm}

\begin{proof}  Define 
\[ g(u):= \sqrt{ \frac{2}{p-1}} \left( \frac{1}{(1-u)^\frac{p-1}{2}}-1 \right).\]  
Note that this choice of $g$ is different from the one used above, as it is easier to manage. It does verify the conditions of Lemma 3.2 and therefore one has $ -\Delta u \ge g(u)$ a.e. in $ \Omega$,  and by Lemma 3.4 we have 
\begin{equation} \label{fff}
 \int_\Omega g(u) H(u) dx \le \int_\Omega f(u) dx,
 \end{equation} where $ H(u):= \int_0^u f''(\tau) g( \tau) d \tau.$  A computation shows that 
\[ H(u) = C_p \left( \frac{1}{(1-u)^\frac{3p+1}{2}} -1 \right) + \tilde{C}_p \left( 1 - \frac{1}{(1-u)^{p+1}} \right) \] where $ C_p, \tilde{C}_p>0$.   Now writing out (\ref{fff}) one obtains an estimate of the form 
\[ \int_\Omega \frac{1}{(1-u)^{2p}} dx \le C(p) \int_\Omega \frac{1}{(1-u)^\frac{3p+1}{2}} dx + C(p) \int_\Omega \frac{1}{(1-u)^p} dx.\]
Since $p>1$, we have that $ \frac{3p+1}{2} < 2p$, from which one easily obtains the desired result. 
\end{proof}

\medskip \noindent {\bf Acknowledgements:} This work began while the second author was visiting the Pacific Institute for the Mathematical Sciences, for which he thanks Dr.  Ghoussoub for the kind invitation and hospitality.

\end{document}